\documentclass[11pt,a4paper]{article}
\usepackage[T1]{fontenc}
\usepackage{lmodern}
\usepackage{amsmath,amssymb,amsthm,mathtools}
\usepackage[margin=29mm]{geometry}
\usepackage{microtype}
\usepackage{enumitem}
\usepackage[hidelinks]{hyperref}
\hypersetup{pdftitle={Sharp lower bounds for the distribution function of the planar discrepancy function and the Lp-quasinorm for all p>0},pdfauthor={Aicke Hinrichs}}
\newtheorem{theorem}{Theorem}
\newtheorem{lemma}[theorem]{Lemma}
\newtheorem{proposition}[theorem]{Proposition}
\newtheorem{corollary}[theorem]{Corollary}
\theoremstyle{definition}

\theoremstyle{remark}
\newtheorem{remark}[theorem]{Remark}
\numberwithin{equation}{section}
\newcommand{\E}{\mathbb E}
\newcommand{\Prob}{\mathbb P}
\newcommand{\ind}{\mathbf 1}
\newcommand{\D}{\mathcal D}
\newcommand{\B}{\mathcal B}

\allowdisplaybreaks[1]

\title{A Sharp Distributional Lower Bound for Planar Discrepancy}

\author{Aicke Hinrichs\\
Jena\\
\small\href{mailto:aicke.hinrichs@gmail.com}
{\texttt{aicke.hinrichs@gmail.com}}}

\date{\today}

\begin{document}

\maketitle

\begin{abstract}
We prove that the unnormalized discrepancy function of every $n$-point set in the unit square has absolute value at least $c\sqrt{1+\log n}$ on a set of measure at least $\eta$, with absolute constants $c,\eta>0$. Consequently, its $L_p$-quasinorm is bounded below by
$c\,\eta^{1/p} \sqrt{1+\log n}$ for every $p>0$. 
\end{abstract}

\section{Introduction and main results}

Let $\mathcal P_n$ be a set of $n\ge1$ points in $[0,1]^2$. Its
unnormalized discrepancy function is
\begin{equation}\label{eq:disc-def}
 D_{\mathcal P_n}(x,y)
 =\#\bigl(\mathcal P_n\cap([0,x]\times[0,y])\bigr)-nxy.
\end{equation}
The main purpose of this note is to show that
$|D_{\mathcal P_n}|$ exceeds a constant multiple of
$\sqrt{1+\log n}$ on a set of uniformly positive measure.
To state this precisely, define, for $u>0$,
\[
 \Phi_{\mathcal P_n}(u)
   =\lambda_2\{(x,y):|D_{\mathcal P_n}(x,y)|\ge u\},
\]
where $\lambda_2$ denotes Lebesgue measure on $[0,1]^2$.
Throughout, $\log$ denotes the natural logarithm. 
Our main result is
\begin{theorem}\label{thm:disc}
There exist constants $c,\eta>0$ such that
for every $n\ge1$ and every $n$-point set $\mathcal P_n\subset[0,1]^2$,
\begin{equation}\label{eq:disc-result}
 \Phi_{\mathcal P_n} \left( c \sqrt{1+\log n} \right) \ge\eta.
\end{equation}
\end{theorem}
An immediate consequence is the following lower bound for the
$L_p$-(quasi-)norm, defined for $p>0$ by
\[
  \|f\|_{L_p}
  =\left(\int_{[0,1]^2}|f|^p\,d\lambda_2\right)^{1/p}.
\]
\begin{corollary}\label{cor:lp}
For every $n\ge1$, every $n$-point set
$\mathcal P_n\subset[0,1]^2$, and every $p>0$,
\[
 \|D_{\mathcal P_n}\|_{L_p}
   \ge c\,\eta^{1/p}\sqrt{1+\log n},
\]
where $c$ and $\eta$ are the constants from
Theorem~\ref{thm:disc}.
\end{corollary}

For historical background and further references on discrepancy
bounds in various settings and dimensions, we refer to the survey
\cite{Bilyk}. Corollary~\ref{cor:lp} confirms
\cite[Conjecture~3.3]{Bilyk} in the planar case.

We briefly recall the classical results in the planar case. Roth proved the planar lower bound of order $\sqrt{1+\log n}$ for the $L_2$-norm in \cite{Roth}. His proof also gives the sharp lower bound of order $(1+\log n)^{\frac{d-1}{2}}$ for the $L_2$-norm in general dimension $d$. Schmidt extended this lower bound of order $\sqrt{1+\log n}$ to the $L_p$-norm in the range $1<p<2$ in \cite{Schmidt}, also with the corresponding lower bound in general dimension. 
For $p\ge2$, the same lower bound follows directly from Roth's
$L_2$ estimate by monotonicity of the $L_p$-norms.
Hal\'asz proved the corresponding estimate for the $L_1$-norm in \cite{Halasz}. His argument does not yield the corresponding lower bound of order
$(1+\log n)^{(d-1)/2}$ in dimensions $d>2$. Obtaining the corresponding $L_1$ lower bound of order
$(1+\log n)^{(d-1)/2}$ in dimensions $d\ge3$ remains
a major open problem in discrepancy theory.

The classical construction of Davenport in \cite{Davenport} gives a sequence of point sets in $[0,1]^2$ with discrepancy of order $\sqrt{1+\log n}$ for the $L_2$-norm. By monotonicity of the $L_p$-(quasi-)norms, the same point sets satisfy an upper bound of order $\sqrt{1+\log n}$ for every $0<p\le2$. This proves optimality for $0<p\le2$. For the remaining range $p>2$, Chen~\cite{Chen} proved
that, for every $n\ge1$, there exists an $n$-point set
$\mathcal P_n\subset[0,1]^2$ satisfying
\[
 \|D_{\mathcal P_n}\|_{L_p}\le C_p\sqrt{1+\log n}.
\]
Thus Corollary~\ref{cor:lp} is optimal in its order of
growth in $n$ for every fixed $p>0$.
For a discussion of suitable modifications of the
van der Corput point set, see \cite[Section~6]{Bilyk}.
Similarly,
the $L_2$ upper bound and Chebyshev's inequality show that the scale
$\sqrt{1+\log n}$ in Theorem~\ref{thm:disc} cannot be replaced
by a function of larger order while retaining a uniform positive
lower bound for the measure.

The paper is organized as follows. In Section~\ref{sec:haar}, we recall necessary notation for the orthogonal system of Haar functions and summarize the needed properties of Haar coefficients of the discrepancy function. A key observation in lower bound proofs is that if the interior of a dyadic rectangle contains no point of $\mathcal P_n$,
then the corresponding Haar coefficient depends only on $n$
and the area of the rectangle. To keep the paper self-contained and clarify which ingredients are classical and which are new, we include concise versions of several known proofs.

A useful, though not entirely accurate, heuristic is to view the
planar discrepancy function as a sum of about $\log n$
independent Rademacher variables. The main technical problem is to make this heuristic precise and control the deviation from the idealized model. Section~\ref{sec:lower-bounds} revisits the classical lower bounds from this
probabilistic viewpoint, states the distributional estimate
for Haar polynomials in Theorem~\ref{thm:haar}, and uses it
to prove Theorem~\ref{thm:disc}.
Section~\ref{sec:averaging} establishes the averaging lemma,
Lemma~\ref{lem:averaging}, and then applies it to prove
Theorem~\ref{thm:haar}.

The classical arguments in Subsections~\ref{subsec:roth}--\ref{subsec:halasz} provide background
and motivation but are not needed for the proof of
Theorem~\ref{thm:disc}. The principal new ingredients are the
distributional estimate in Theorem~\ref{thm:haar} and the
averaging lemma, Lemma~\ref{lem:averaging}.

\section{Haar functions and discrepancy coefficients}
\label{sec:haar}

Let $\D$ be the collection of dyadic intervals
\[
I=[j2^{-a},(j+1)2^{-a}),\qquad
a\in\mathbb N_0,\quad j\in\{0,\ldots,2^a-1\}.
\]
For $I\in\D$, let $|I|$ denote its length and 
$I_{\mathrm{left}}$ and $I_{\mathrm{right}}$ 
its left and right halves. Define the Haar function with support $I$ by
\[
 h_I=\ind_{I_{\mathrm{right}}}-\ind_{I_{\mathrm{left}}}.
\]
For a dyadic rectangle $R=I\times J$, put
$h_R(x,y)=h_I(x)h_J(y)$ and $|R|=|I||J|=\lambda_2(R)$.
These tensor Haar functions are pairwise orthogonal and satisfy
\[
 \|h_R\|_2^2=|R|.
\]
For $f\in L_1([0,1]^2)$, we use the Haar expansion coefficients
\[
 a_R(f)=\frac{\langle f,h_R\rangle}{|R|}.
\]
Thus $a_R(f)$ is distinguished from the unnormalized inner product
$\langle f,h_R\rangle$.

Together with the constant function and the one-variable functions
$h_I(x)$ and $h_J(y)$, the functions $h_R$ form a complete
orthogonal system in $L_2([0,1]^2)$. In particular, every real
$f\in L_2([0,1]^2)$ has a decomposition
\begin{equation}\label{eq:haar-expansion}
 f(x,y)=u(x)+v(y)+\sum_{R\in\D\times\D}a_R(f)h_R(x,y),
\end{equation}
with convergence in $L_2$. Here $u,v\in L_2([0,1])$
contain the constant and one-variable components;
the allocation of the constant term does not matter.

For integers $k\ge0$ and $0\le a\le k$, define
\[
 \D_{k,a}=\{I\times J:I,J\in\D,\ |I|=2^{-a},\
                              |J|=2^{-(k-a)}\}.
\]
Each collection $\D_{k,a}$ partitions $[0,1)^2$ into
$2^k$ rectangles of area $2^{-k}$. We call $a$ the shape
index and $k$ the level. We retain the half-open convention
when assigning points to rectangles. Identities between
functions are understood almost everywhere; in particular,
their values on the union of all dyadic boundaries,
a set of Lebesgue measure zero, are immaterial.

\subsection{The coefficients of the discrepancy function}

The following well-known calculation separates the contribution of the points
from that of the volume term.

\begin{lemma}\label{lem:disc-coeff}
Let $R=I\times J$ be a dyadic rectangle and put
$\tau_I(t)=\int_t^1h_I(x)\,dx$. Then
\begin{equation}\label{eq:disc-coefficient}
 a_R(D_{\mathcal P_n})
 =\frac1{|R|}\sum_{p\in\mathcal P_n}
                 \tau_I(p_1)\tau_J(p_2)-\frac{n|R|}{16}.
\end{equation}
In particular, if the interior of $R$ contains no point of
$\mathcal P_n$, then
\begin{equation}\label{eq:empty}
 a_R(D_{\mathcal P_n})=-\frac{n|R|}{16}.
\end{equation}
\end{lemma}

\begin{proof}
The contribution of $p=(p_1,p_2)$ to the counting function is
$\ind_{[p_1,1]}(x)\ind_{[p_2,1]}(y)$. Hence Fubini's theorem gives
\[
 \left\langle
 \ind_{[p_1,1]}\otimes\ind_{[p_2,1]},h_R
 \right\rangle=\tau_I(p_1)\tau_J(p_2).
\]
Moreover,
\[
 \int_0^1x h_I(x)\,dx=\frac{|I|^2}{4},\qquad
 \langle nxy,h_R\rangle=\frac{n|R|^2}{16}.
\]
Dividing the difference by $|R|$ proves
\eqref{eq:disc-coefficient}. The function $\tau_I$ vanishes outside
the interior of $I$, including at its endpoints. Consequently every
point contribution vanishes when the interior of $R$ is empty,
which proves \eqref{eq:empty}.
\end{proof}

Choose the unique integer $\ell\ge2$ such that
\begin{equation}\label{eq:n-ell}
 2^{\ell-2}\le n<2^{\ell-1},\qquad
 \rho=n2^{-\ell}\in[1/4,1/2).
\end{equation}
For each $k\ge\ell$ and $0\le a\le k$, let
\[
 \B_{k,a}
 =\{R\in\D_{k,a}:R\cap\mathcal P_n\ne\varnothing\}
\]
be the collection of bad rectangles. The remaining rectangles
in $\D_{k,a}$ are called good. Since rectangles of a fixed
shape and level are disjoint,
\begin{equation}\label{eq:disc-bad-count}
 \#\B_{k,a}\le n.
\end{equation}
In particular, the union of the bad rectangles of any principal
shape has measure at most $n2^{-\ell}<1/2$.
We use the half-open convention when designating rectangles
as good or bad, although the coefficient formula only requires
an empty interior.

A dyadic rectangle $Q$ is an ancestor of $R$ if $R\subseteq Q$;
equality is allowed, and the two side lengths may be enlarged
independently. Empty rectangles have only empty descendants.
Thus the good rectangles are hereditary: every descendant of
a good rectangle is good. Equivalently, every ancestor of a bad
rectangle among the levels under consideration is bad.

For the rescaled discrepancy function
\begin{equation}\label{eq:disc-rescaling}
 f=-\frac{16}{\rho}D_{\mathcal P_n},
\end{equation}
Lemma~\ref{lem:disc-coeff} yields
\begin{equation}\label{eq:disc-good-coeff}
 a_R(f)=2^{-s}
 \qquad
 \bigl(s\ge0,\quad
 R\in\D_{\ell+s,a}\setminus\B_{\ell+s,a}\bigr).
\end{equation}
The coefficients on bad rectangles need not satisfy this identity.

\subsection{The regular sums and the Rademacher model}

Define the regular sum at level $k$ by
\begin{equation}\label{eq:H-definition}
 H_k=\sum_{a=0}^k\sum_{R\in\D_{k,a}}h_R.
\end{equation}
Let $\varepsilon_j$ be the $j$th Rademacher function on $[0,1)$, with value
$-1$ when the $j$th binary digit of the argument is zero and $+1$ when it is one.
Summing the Haar functions over all rectangles of one shape gives
\begin{equation}\label{eq:H-rademacher}
 H_k(x,y)=\sum_{a=0}^k
              \varepsilon_{a+1}(x)\varepsilon_{k-a+1}(y).
\end{equation}
For fixed $k$, the summands involve disjoint pairs of independent
binary digits. They are therefore independent Rademacher variables
on the square equipped with Lebesgue probability measure. In
particular,
\begin{equation}\label{eq:H-energy}
 \|H_k\|_2^2=k+1.
\end{equation}
Different levels are orthogonal, although they need not be independent.

To identify precisely how the discrepancy function differs from
these regular sums, let $g$ consist of the constant and one-variable
components of $f$, together with its mixed Haar terms of levels
$k<\ell$. For $R\in\D_{\ell+s,a}$ put
$\beta_R=a_R(f)-2^{-s}$. The truncation of the expansion of $f$
at level $\ell+S$ can then be written as
\begin{equation}\label{eq:disc-truncated-decomp}
 \begin{split}
 f_S&=g+\sum_{s=0}^S2^{-s}H_{\ell+s}+\mathcal C_S,\\
 \mathcal C_S&=\sum_{s=0}^S\sum_{a=0}^{\ell+s}
                     \sum_{R\in\B_{\ell+s,a}}\beta_Rh_R.
 \end{split}
\end{equation}
Since $f\in L_2([0,1]^2)$, completeness and orthogonality imply
$f_S\to f$ in $L_2$ as $S\to\infty$.

The principal regular sum $H_\ell$ is exactly a sum of $\ell+1$
independent signs, and $\ell+1$ is comparable to $1+\log n$.
The remaining terms have different roles. The coarser part $g$
has a restricted dependence on the binary digits, the regular finer
levels have geometrically decreasing coefficients, and the
corrections $\mathcal C_S$ occur only on bad rectangles, whose ancestors among the levels under consideration are also bad. Their coefficients may be large, so they cannot simply
be discarded as a small error. These features underlie the distributional estimate stated in
Section~\ref{sec:lower-bounds} and proved in
Section~\ref{sec:averaging}.

\section{From classical lower bounds to a distributional estimate}
\label{sec:lower-bounds}

We first revisit the classical planar lower bounds of Roth, Schmidt,
and Hal\'asz \cite{Roth,Schmidt,Halasz}; see also \cite{Bilyk}.
The good rectangles at the principal Haar level supply the
essential contribution in each case. Orthogonality gives the
$L_2$ bound, the product Littlewood--Paley inequality yields
the $L_p$ bounds for $p>1$, and a bounded complex Riesz product
gives the $L_1$ bound. We then turn
from these norm estimates to the distributional statement needed
for Theorem~\ref{thm:disc}.

\subsection{Roth's bound by orthogonality}
\label{subsec:roth}

Keep the principal level $\ell$ and the number $\rho$ from
\eqref{eq:n-ell}, and put $m=\ell+1$.
For every good rectangle $R$ in the principal level,
Lemma~\ref{lem:disc-coeff} gives
\[
 a_R(D_{\mathcal P_n})=-\frac{\rho}{16}.
\]
Each principal shape contains at least $2^{\ell-1}$ good
rectangles. Hence orthogonality and Bessel's inequality yield
\[
 \|D_{\mathcal P_n}\|_2^2
 \ge \sum_{a=0}^{\ell}
      \sum_{R\in\D_{\ell,a}\setminus\B_{\ell,a}}
          |a_R(D_{\mathcal P_n})|^2\,|R|
 \ge \frac{\rho^2m}{512}.
\]
Since $\rho\ge1/4$, we obtain Roth's planar lower bound
\begin{equation}\label{eq:roth-planar}
 \|D_{\mathcal P_n}\|_2
 \ge\frac{\rho\sqrt m}{16\sqrt2}
 \ge\frac{\sqrt m}{64\sqrt2}.
\end{equation}

\subsection{Schmidt's bound via the Littlewood--Paley inequality}
\label{subsec:schmidt}

For $f\in L_p([0,1]^2)$, define its mixed dyadic square function by
\[
 Sf=\left(
       \sum_{R\in\D\times\D}|a_R(f)|^2\ind_R
     \right)^{1/2}.
\]
For every $1<p<\infty$, the product Littlewood--Paley
inequality gives a constant $A_p>0$ such that
\begin{equation}\label{eq:littlewood-paley}
 \|f\|_p\ge A_p\|Sf\|_p;
\end{equation}
see \cite[Section~3]{Bilyk}.
To account explicitly for the one-variable terms, let $E_x$ and
$E_y$ denote averaging in the respective coordinates and put
$P=(I-E_x)(I-E_y)$. Then $Sf=S(Pf)$, and the product
Littlewood--Paley inequality applied to $Pf$ gives
$\|Sf\|_p\lesssim_p\|Pf\|_p\le4\|f\|_p$,
which yields \eqref{eq:littlewood-paley}.

Let
\[
 M=\sum_{a=0}^{\ell}
       \sum_{R\in\D_{\ell,a}\setminus\B_{\ell,a}}\ind_R
\]
count the good principal rectangles containing a given point.
Since at least half of the rectangles of each shape are good,
\[
 0\le M\le m,\qquad \int_{[0,1]^2}M\,d\lambda_2\ge\frac m2.
\]
Consequently, the set $E=\{M\ge m/4\}$ has measure at least
$1/4$: indeed,
\[
 \frac m2\le\int M\,d\lambda_2
 \le\frac m4+m\lambda_2(E).
\]
On $E$, the good principal coefficients give
\[
 S(D_{\mathcal P_n})
 \ge\frac{\rho}{16}\sqrt M
 \ge\frac{\rho\sqrt m}{32}.
\]
Applying \eqref{eq:littlewood-paley}, we obtain
\begin{equation}\label{eq:schmidt-planar}
 \|D_{\mathcal P_n}\|_p
 \ge A_p\|S(D_{\mathcal P_n})\|_p
 \ge\frac{A_p\rho}{32}\,4^{-1/p}\sqrt m
 \gtrsim_p\sqrt{1+\log n}.
\end{equation}
This recovers Schmidt's planar lower bound for every
$1<p<\infty$. The restriction on $p$ enters only through
the Littlewood--Paley inequality.

\subsection{Hal\'asz's bound via a complex Riesz product}
\label{subsec:halasz}

For $0\le a\le\ell$, define
\[
 q_a=-\sum_{R\in\D_{\ell,a}\setminus\B_{\ell,a}}h_R.
\]
These functions satisfy $|q_a|\le1$. By
Lemma~\ref{lem:disc-coeff} and the count of good rectangles,
\begin{equation}\label{eq:classical-pairing}
 \langle D_{\mathcal P_n},q_a\rangle
 =\frac{\rho}{16}
   \sum_{R\in\D_{\ell,a}\setminus\B_{\ell,a}}|R|
 \ge\frac{\rho}{32}.
\end{equation}
Following Hal\'asz \cite{Halasz}, we combine these Haar
polynomials into a bounded complex Riesz product.

We first record an estimate for its higher-order terms.
For every $k\ge0$ and $0\le a\le k$,
\begin{equation}\label{eq:shape-pairing}
 \sum_{R\in\D_{k,a}}
       |\langle D_{\mathcal P_n},h_R\rangle|
 \le\frac{5n2^{-k}}{16}.
\end{equation}
Indeed, $0\le\tau_I\le|I|/2$, and each point contributes
to at most one rectangle of a given shape and level.
By \eqref{eq:disc-coefficient}, the point contributions
therefore sum to at most $n2^{-k}/4$, while the absolute
volume contributions sum to $n2^{-k}/16$.

For distinct indices $a_1<\cdots<a_r$, with $r\ge2$,
the product $q_{a_1}\cdots q_{a_r}$ is a Haar sum over
rectangles with side lengths $2^{-a_r}$ and
$2^{-(\ell-a_1)}$, with coefficients in $\{-1,0,1\}$.
To see this, expand the product and discard empty
intersections. On each remaining intersection, the finest
Haar function in each coordinate is multiplied only by
functions constant on its support. The intersections are
disjoint rectangles of the stated shape.
Consequently, \eqref{eq:shape-pairing} gives
\begin{equation}\label{eq:product-pairing}
 |\langle D_{\mathcal P_n},q_{a_1}\cdots q_{a_r}\rangle|
 \le\frac{5\rho}{16}\,2^{-(a_r-a_1)}.
\end{equation}

Set $\gamma=1/8$, $t=\gamma/\sqrt m$, and
\[
 \Psi=\operatorname{Im}\prod_{a=0}^{\ell}(1+itq_a).
\]
Since $|q_a|\le1$,
\begin{equation}\label{eq:riesz-bounded}
 \|\Psi\|_\infty
 \le(1+t^2)^{m/2}\le e^{\gamma^2/2}.
\end{equation}
The linear term is $t\sum_{a=0}^{\ell}q_a$, and its
pairing with the discrepancy function is at least
$\rho mt/32$ by \eqref{eq:classical-pairing}.
The remaining terms have odd degree at least three.

For a fixed span $d=a_r-a_1$, there are $m-d$ choices
of the endpoints and $\binom{d-1}{r-2}$ choices of the
interior indices. Thus \eqref{eq:product-pairing} yields
\begin{align*}
 \left|\left\langle D_{\mathcal P_n},
        \Psi-t\sum_{a=0}^{\ell}q_a\right\rangle\right|
 &\le\frac{5\rho}{16}
   \sum_{d=1}^{m-1}(m-d)2^{-d}
       \sum_{r=3}^{d+1}\binom{d-1}{r-2}t^r\\
 &\le\frac{5\rho m}{16}\,t^2
   \sum_{d=1}^{\infty}2^{-d}
       \bigl((1+t)^{d-1}-1\bigr)\\
 &=\frac{5\rho m}{16}\frac{t^3}{1-t}.
\end{align*}
Here we bounded the sum over odd degrees by the sum
over all degrees at least three.
Since $t\le1/8$,
\[
 \frac{10t^2}{1-t}\le\frac5{28}<\frac12.
\]
It follows that
\[
 \langle D_{\mathcal P_n},\Psi\rangle
 \ge\frac{\rho mt}{64}.
\]
Using \eqref{eq:riesz-bounded}, we obtain
\begin{equation}\label{eq:halasz-planar}
 \|D_{\mathcal P_n}\|_1
 \ge\frac{\langle D_{\mathcal P_n},\Psi\rangle}
          {\|\Psi\|_\infty}
 \ge\frac{\rho\gamma}{64e^{\gamma^2/2}}\sqrt m
 \gtrsim\sqrt{1+\log n}.
\end{equation}

\subsection{A distributional estimate for Haar expansions}

The preceding arguments establish the lower bounds for $p\ge1$,
but do not directly extend to $0<p<1$. We now state a stronger distributional estimate based on the
decomposition from Section~\ref{sec:haar}, deferring its proof
to Section~\ref{sec:averaging}. The following theorem isolates its
essential features: prescribed coefficients on hereditary good
rectangles and arbitrary coefficients on bad rectangles.

Fix a principal level $\ell\ge0$ and a finite depth $S\ge0$. 
For $\ell\le k\le\ell+S$
and $0\le a\le k$, designate a collection
$\B_{k,a}\subseteq\D_{k,a}$ of bad rectangles.
The remaining rectangles are called good.

\begin{theorem}\label{thm:haar}
Let $0\le\delta<1$. Suppose that the good rectangles are
hereditary among the levels under consideration and
\begin{equation}\label{eq:principal-count}
 \#\B_{\ell,a}\le\delta\,2^\ell
       \qquad(0\le a\le\ell).
\end{equation}
Let
\begin{equation}\label{eq:f-form}
 f=g+\sum_{s=0}^{S}\sum_{a=0}^{\ell+s}
                    \sum_{R\in\D_{\ell+s,a}}\alpha_Rh_R,
\end{equation}
where $g$ is the sum of a finite real linear combination of
tensor Haar functions supported on rectangles of area greater
than $2^{-\ell}$ and arbitrary real measurable functions
$u(x)+v(y)$, finite almost everywhere. 
Assume
\begin{equation}\label{eq:good-coeff}
 \alpha_R=2^{-s}\quad\text{whenever }R\in\D_{\ell+s,a}
                                  \setminus\B_{\ell+s,a}.
\end{equation}
The bad coefficients are arbitrary real numbers. 
Then, for every integer $K\ge0$,
\begin{equation}\label{eq:variable-bound}
 \lambda_2\left\{|f|\ge
       \frac{\sqrt{(1-\delta)(\ell+1)}}{8\sqrt{K+1}}\right\}
 \ge \frac{1-\delta}{25}
       -\frac{64(K+1)(3K+7)}{9(1-\delta)4^K}.
\end{equation}
In particular, when $\delta\le1/2$,
\begin{equation}\label{eq:explicit-haar}
 \lambda_2\left\{|f|\ge
             \frac{\sqrt{\ell+1}}{8\sqrt{22}}\right\}
             \ge\frac1{100}.
\end{equation}
\end{theorem}

\begin{remark}\label{rem:counts}
Apart from heredity, no restriction is imposed on the finer
bad collections. For every fixed $\delta<1$, the right-hand
side of \eqref{eq:variable-bound} is positive for all
sufficiently large $K$. Thus the theorem gives a positive
measure bound whenever the principal bad proportion is
bounded away from one.
\end{remark}

To apply Theorem~\ref{thm:haar} to the full Haar expansion of the
discrepancy function, we will pass to the limit in its finite
truncations using the following elementary observation.

\begin{proposition}\label{prop:limits}
Let $f_q\to f$ in measure on $[0,1]^2$. If, for some
$z>0$ and $\eta>0$,
\[
 \lambda_2\{|f_q|\ge z\}\ge\eta
 \qquad\text{for every }q,
\]
then
\[
 \lambda_2\{|f|\ge z\}\ge\eta.
\]
In particular, for fixed $\ell$ and $\delta$, every bound in
Theorem~\ref{thm:haar} passes to limits in measure of finite
expansions satisfying its hypotheses.
\end{proposition}
\begin{proof}
Suppose $f_q\to f$ in measure and
$\lambda_2\{|f_q|\ge z\}\ge\eta$ for every $q$.
For every $0<\varepsilon<z$,
\[
 \lambda_2\{|f|\ge z-\varepsilon\}
 \ge\eta-\lambda_2\{|f_q-f|>\varepsilon\}.
\]
Letting $q\to\infty$ and then $\varepsilon\downarrow0$
gives $\lambda_2\{|f|\ge z\}\ge\eta$.
\end{proof}

\subsection{Application to the discrepancy function}

\begin{proof}[Proof of Theorem~\ref{thm:disc}]
Use the level $\ell$, the number $\rho$, and the rescaled
function $f=-16D_{\mathcal P_n}/\rho$ from Section~\ref{sec:haar}.
As shown there, the good rectangles are hereditary,
the principal bad proportion is less than $1/2$, and
the coefficients on good rectangles have the values
required by Theorem~\ref{thm:haar}.

The truncations $f_S$ from \eqref{eq:disc-truncated-decomp}
therefore satisfy \eqref{eq:explicit-haar}.
Since $f_S\to f$ in $L_2$, hence in measure,
Proposition~\ref{prop:limits} gives
\[
 \lambda_2\left\{|f|\ge
       \frac{\sqrt{\ell+1}}{8\sqrt{22}}\right\}
 \ge\frac1{100}.
\]
Rescaling and using $\rho\ge1/4$, we obtain
\begin{equation}\label{eq:disc-ell-bound}
 \lambda_2\left\{|D_{\mathcal P_n}|\ge
       \frac{\sqrt{\ell+1}}{512\sqrt{22}}\right\}
 \ge\frac1{100}.
\end{equation}
Finally, \eqref{eq:n-ell} implies
\[
 \log n<(\ell-1)\log2\le\ell,
 \qquad 1+\log n\le\ell+1.
\]
This proves Theorem~\ref{thm:disc}, with
$c=1/(512\sqrt{22})$ and $\eta=1/100$.
\end{proof}

No attempt has been made to optimize these constants.

\section{Proof of the distributional estimate}
\label{sec:averaging}

\subsection{An averaging lemma with hereditary interactions}
\label{subsec:averaging}

The following lemma gives a distributional lower bound for a sum
whose principal terms are products $X_jY_j$ of independent random
signs. We allow additional terms with a prescribed dependence on
these signs, including interactions that vanish whenever they
cross a good position. The proof averages successively over the
$Y$ variables and uses an $L_2$ contraction at each good position.

The additional random variable $W$ allows the terms and the
good positions to depend on randomness other than the displayed
signs. It is independent of those signs and may be viewed as a
parameter when taking the finite averages over them. In the
application to Haar expansions, $W$ collects the binary digits
not selected as $X_j$ or $Y_j$. The estimates are then integrated
over $W$; the bounds on the probabilities of bad positions are
required only for the full probability space, not for each
fixed value of $W$.

We write $X=(X_1,\ldots,X_N)$, $Y=(Y_1,\ldots,Y_N)$,
$X_{\le j}=(X_1,\ldots,X_j)$ and
$Y_{>j}=(Y_{j+1},\ldots,Y_N)$, with the analogous conventions for
strict initial and closed terminal segments. A segment may be empty.

\begin{lemma}\label{lem:averaging}
Let $N\ge1$, and let $X_1,\ldots,X_N,Y_1,\ldots,Y_N$ be independent
uniform signs, independent also of $W$. Suppose the real random variable
$F$ has the representation
\begin{equation}\label{eq:ordered}
 F=A_0(W,Y)+\sum_{j=1}^N\left(
 A_j(W,X_{\le j},Y_{>j})+X_jY_j
 +\sum_{k=1}^j C_{k,j}(W,X_{\le j},Y_{\ge k})\right).
\end{equation}
Let $\mathcal W$ denote the measurable space in which $W$ takes
values. The functions
\[
\begin{aligned}
 A_0&:\mathcal W\times\{-1,1\}^{N}\longrightarrow\mathbb R,\\
 A_j&:\mathcal W\times\{-1,1\}^{j}
                    \times\{-1,1\}^{N-j}\longrightarrow\mathbb R
                    &&(1\le j\le N),\\
 C_{k,j}&:\mathcal W\times\{-1,1\}^{j}
                       \times\{-1,1\}^{N-k+1}\longrightarrow\mathbb R
                    &&(1\le k\le j\le N)
\end{aligned}
\]
are measurable. Their arguments are ordered as in
\eqref{eq:ordered}, and $\{-1,1\}^{0}$ is understood to be
a singleton.
Fix $0\le\delta<1$. Let $\Omega_i$ be measurable
with respect to $\sigma(W,X_{<i},Y_{>i})$, with
$\Prob(\Omega_i)\le\delta$. Suppose
\begin{equation}\label{eq:crossing-vanish}
 \ind_{\Omega_i^c}C_{k,j}=0
              \qquad(1\le k\le i\le j\le N).
\end{equation}
We call position $i$ good on $\Omega_i^c$ and bad on $\Omega_i$.
Thus an interaction $C_{k,j}$ vanishes whenever a position
between $k$ and $j$, inclusive, is good.
Then, for every real $t$,
\begin{equation}\label{eq:char}
 |\E e^{itF}|^2\le
 \delta+(1-\delta)\kappa(t)^{2N},
 \qquad \kappa(t)=\max\{|\cos t|,|\sin t|\}.
\end{equation}
Furthermore,
\begin{equation}\label{eq:lemma-measure}
 \Prob\left\{|F|\ge\frac{\sqrt{(1-\delta)N}}4\right\}
              \ge\frac{1-\delta}{25}.
\end{equation}
\end{lemma}

The proof has a martingale flavour. Averaging over successive
$Y$ variables gives conditional expectations with respect to
a decreasing family of $\sigma$-algebras. At each stage, we
average only up to the last good position. The event selecting
that position does not depend on the signs being averaged out,
so this choice preserves expectations. An $L_2$ contraction
at good positions, together with preservation of the squared
modulus at bad positions, then yields the required estimate.

\begin{proof}
Let $F_j$ denote the partial sum in \eqref{eq:ordered} through index
$j$, with $F_0=A_0(W,Y)$. Let $P_l$ average only
$Y_1,\ldots,Y_l$, keeping $W$, all $X$ variables, and $Y_{>l}$ fixed;
$P_0$ is the identity operator.
At each step, we average the $Y$ variables up to the last
good position. This preserves the expectation while allowing
an $L_2$ contraction at each new good position.
For $0\le j\le N$, define
\begin{equation}\label{eq:last-good}
 L_j=\max\bigl(\{i\in\{1,\ldots,j\}:\Omega_i^c\text{ occurs}\}
                        \cup\{0\}\bigr).
\end{equation}
If $l\ge1$, the event $\{L_j=l\}$ equals
$\Omega_l^c\cap\bigcap_{i=l+1}^j \Omega_i$; for $l=0$ it equals
$\bigcap_{i=1}^j\Omega_i$. Hence
\begin{equation}\label{eq:last-measurable}
 \{L_j=l\}\in\sigma(W,X_{\le j},Y_{>l}).
\end{equation}
In particular, this event does not depend on the $Y$ variables
averaged by $P_l$.

Define
\begin{equation}\label{eq:V}
 V_j=\sum_{l=0}^j\ind_{\{L_j=l\}}P_l e^{itF_j},
 \qquad M_j=\E|V_j|^2.
\end{equation}
These quantities are well-defined even when $F$ is not integrable,
since its complex exponential has modulus one. We have
\begin{equation}\label{eq:V-properties}
 |V_j|\le1,\qquad \E V_j=\E e^{itF_j},\qquad M_0=1.
\end{equation}
The middle equality follows from \eqref{eq:last-measurable} by
averaging separately on each event $\{L_j=l\}$.

We prove the recurrence
\begin{equation}\label{eq:recurrence}
 M_j\le\kappa(t)^2M_{j-1}+(1-\kappa(t)^2)\delta.
\end{equation}

\emph{Bad positions.}
On $\Omega_j\cap\{L_{j-1}=l\}$ one has $L_j=l<j$.
For $l\ge1$, the position $l$ is good; therefore
\eqref{eq:crossing-vanish} implies $C_{k,j}=0$ there whenever
$k\le l$. Each remaining term of $F_j-F_{j-1}$ does not depend
on $Y_1,\ldots,Y_l$. The same assertion is vacuous when $l=0$.
Moreover, the event under consideration is itself independent of
$Y_1,\ldots,Y_l$. Factoring the increment out of the corresponding
average in \eqref{eq:V} gives
\begin{equation}\label{eq:bad-identity}
 V_j=e^{it(F_j-F_{j-1})}V_{j-1}\quad\text{on }\Omega_j.
\end{equation}
Thus the bad part preserves the squared modulus:
\begin{equation}\label{eq:bad-energy}
 \E[\ind_{\Omega_j}|V_j|^2]=\E[\ind_{\Omega_j}|V_{j-1}|^2].
\end{equation}

\emph{Good positions.}
On $\Omega_j^c$ all the terms $C_{k,j}$ vanish and $L_j=j$.
Because $\Omega_j$ depends on neither $X_j$ nor $Y_j$, and is also
independent of $Y_{<j}$, it follows that
\begin{equation}\label{eq:good-form}
 V_j=e^{itA_j}\E_{Y_j}\left[
 e^{itX_jY_j}\,P_{j-1}e^{itF_{j-1}}\right]
                 \quad\text{on }\Omega_j^c.
\end{equation}
The following conditional averaging identity is useful:
\begin{equation}\label{eq:tower}
 P_{j-1}e^{itF_{j-1}}=P_{j-1}V_{j-1}.
\end{equation}
To verify it, set $Z=e^{itF_{j-1}}$ and
$\Gamma_l=\{L_{j-1}=l\}$. Since $\Gamma_l$ does not depend on $Y_{\le l}$,
\[
 P_{j-1}(\ind_{\Gamma_l}P_l Z)
 =P_{j-1}P_l(\ind_{\Gamma_l}Z)
 =P_{j-1}(\ind_{\Gamma_l}Z).
\]
Summing over $l=0,\ldots,j-1$ proves \eqref{eq:tower}.

For a function $v$ on $\{-1,1\}$, define
\[
 (\mathsf K_tv)(x)=\frac12\sum_{y=\pm1}e^{itxy}v(y).
\]
On the two-dimensional Hilbert space with uniform probability measure,
its matrix is
\[
 \frac12\begin{pmatrix}e^{it}&e^{-it}\\e^{-it}&e^{it}\end{pmatrix}.
\]
The vectors $(1,1)$ and $(1,-1)$ are orthogonal eigenvectors, with
eigenvalues $\cos t$ and $i\sin t$, respectively. Hence
\begin{equation}\label{eq:operator}
 \|\mathsf K_t\|_{2\to2}=\kappa(t).
\end{equation}
For fixed $(W,X_{<j},Y_{>j})$, the function
$P_{j-1}V_{j-1}$ depends on $Y_j$ but not on $X_j$.
Apply \eqref{eq:operator} to \eqref{eq:good-form}, use
\eqref{eq:tower}, and then apply conditional Jensen. This gives
\begin{align}
 \E[\ind_{\Omega_j^c}|V_j|^2]
 &\le\kappa(t)^2\E[\ind_{\Omega_j^c}|P_{j-1}V_{j-1}|^2]\notag\\
 &\le\kappa(t)^2\E[\ind_{\Omega_j^c}|V_{j-1}|^2].
                                                        \label{eq:good-energy}
\end{align}
The independence of $\Omega_j$ from the variables being averaged is
essential in both inequalities.

Combining \eqref{eq:bad-energy} and \eqref{eq:good-energy}, we obtain
\[
 M_j\le\kappa(t)^2M_{j-1}
       +(1-\kappa(t)^2)\E[\ind_{\Omega_j}|V_{j-1}|^2]
 \le\kappa(t)^2M_{j-1}+(1-\kappa(t)^2)\delta.
\]
This proves \eqref{eq:recurrence}. Starting from $M_0=1$, induction gives
\[
 M_N\le\delta+(1-\delta)\kappa(t)^{2N}.
\]
By \eqref{eq:V-properties} and Cauchy--Schwarz,
$|\E e^{itF}|^2=|\E V_N|^2\le M_N$, proving \eqref{eq:char}.

To deduce \eqref{eq:lemma-measure}, put $\theta=1-\delta$ and
$t=1/(2\sqrt N)$. Then $0<t\le1/2$, so $\kappa(t)=\cos t$.
The inequality $\cos t\le e^{-t^2/2}$ on this interval implies
\[
 \kappa(t)^{2N}\le e^{-1/4}<\frac45,
 \qquad |\E e^{itF}|\le\sqrt{1-\theta/5}\le1-\theta/10.
\]
Consequently,
\begin{equation}\label{eq:cos-defect}
 \E(1-\cos(tF))\ge \theta/10.
\end{equation}
Set $z=\sqrt{\theta N}/4$. On $\{|F|<z\}$,
$1-\cos(tF)\le t^2z^2/2=\theta/128$; everywhere it is at most $2$.
Thus \eqref{eq:cos-defect} yields
\[
 \Prob\{|F|\ge z\}\ge\frac12\left(\frac \theta{10}-\frac \theta{128}\right)
 =\frac{59\theta}{1280}\ge\frac \theta{25}.
\]
This completes the proof.
\end{proof}

\begin{remark}
The estimate $\Prob(\Omega_i)\le\delta$ in the lemma is unconditional:
no bound on $\Prob(\Omega_i\mid W)$ is required. All energy estimates
are integrated over $W$. This distinction is needed when the lemma
is applied after retaining only some of the Rademacher variables.
\end{remark}

\subsection{Proof of the Haar distributional estimate}

We now apply Lemma~\ref{lem:averaging} to prove
Theorem~\ref{thm:haar}. We retain suitably separated pairs of
Rademacher variables and include all remaining digits in the
auxiliary variable $W$. The main task is to express the Haar
expansion in the form \eqref{eq:ordered} and verify the
hereditary vanishing condition \eqref{eq:crossing-vanish}.

\begin{proof}[Proof of Theorem~\ref{thm:haar}]
Let $m=\ell+1$ and $\theta=1-\delta$. For $R\in\D_{\ell+s,a}$ put
$\beta_R=\alpha_R-2^{-s}$. Thus $\beta_R=0$ on good rectangles.
Using the regular sums $H_k$ from \eqref{eq:H-definition}, put
\[
 \mathcal C=\sum_{s=0}^S\sum_{a=0}^{\ell+s}
              \sum_{R\in\B_{\ell+s,a}}\beta_Rh_R.
\]
Then
\begin{equation}\label{eq:uniform-decomp}
 f=g+\sum_{s=0}^S2^{-s}H_{\ell+s}+\mathcal C.
\end{equation}
The sum $\mathcal C$ contains every bad correction, without alteration
or truncation.

Recall the Rademacher representation \eqref{eq:H-rademacher}.
For the principal level write
\begin{equation}\label{eq:original-signs}
 \xi_i=\varepsilon_i(x),\qquad \eta_i=\varepsilon_{m+1-i}(y)\quad(1\le i\le m).
\end{equation}
These $2m$ signs are independent and uniform, and
$H_\ell=\sum_{i=1}^m\xi_i\eta_i$.

Fix an integer $K\ge0$ and retain every $(K+1)$st principal
pair of signs $(\xi_i,\eta_i)$. This ensures that
no regular term from the first $K$ finer levels contains
two retained signs. The remaining regular levels will be
estimated in $L_2$.
Specifically, retain the indices
\[
 p_i=1+(i-1)(K+1),\qquad 1\le i\le N,
 \qquad N=\left\lceil\frac{m}{K+1}\right\rceil.
\]
Set $X_i=\xi_{p_i}$ and $Y_i=\eta_{p_i}$.
Let $W$ consist of all the other binary digits of $x$ and $y$,
i.e., all digits excluding those at positions $p_i$ in $x$
and $m+1-p_i$ in $y$, $1\le i\le N$.

Except on the null set of ambiguous binary expansions, this identifies
Lebesgue measure with the product of the retained independent signs
and $W$. It also accommodates the measurable summands $u(x)$ and $v(y)$.

Remove the regular tail only:
\begin{equation}\label{eq:tail-def}
 \mathcal R_K=\sum_{s=K+1}^S2^{-s}H_{\ell+s},\qquad G=f-\mathcal R_K,
\end{equation}
where the sum is zero if $S\le K$.
We verify the representation \eqref{eq:ordered} and
the vanishing condition \eqref{eq:crossing-vanish} for $G$.

\emph{The principal term.}
The retained terms of $H_\ell$ are $\sum_{i=1}^NX_iY_i$.
Its other terms depend only on $W$ and are included in $A_0$.

\emph{The remaining regular finer levels.}
Consider $H_{\ell+s}$ with $1\le s\le\min\{K,S\}$.
A summand in \eqref{eq:H-rademacher} is
$\varepsilon_j(x)\varepsilon_{m+s+1-j}(y)$.
If both factors have indices represented in \eqref{eq:original-signs},
it is $\xi_j\eta_{j-s}$. For $1\le s\le K$, these two distinct
indices cannot both be retained, since distinct retained indices are
at least $K+1$ apart. If a factor has an index outside the principal
range, it is already a component of $W$. Thus every term is constant
or depends on a single retained sign, with a coefficient depending on
$W$. Such terms can be included among the $A_i$: a function of $X_i$
can be assigned to $A_i$, and a function of $Y_i$ to $A_{i-1}$.

\emph{Coarser terms.}
We now consider the terms collected in $g$: the tensor Haar
terms supported on rectangles of area greater than $2^{-\ell}$
and the one-variable summands $u(x)$ and $v(y)$.
For a rectangle with side lengths $2^{-a}$ and $2^{-b}$,
$h_R$ depends only on the first $a+1$ binary digits of $x$ and the
first $b+1$ binary digits of $y$. In the ordering
\eqref{eq:original-signs}, it is therefore a function of $W$ and the signs
\begin{equation}\label{eq:rectangle-dependence}
 \{X_i:p_i\le a+1\}\quad\text{and}\quad
 \{Y_i:p_i\ge m-b\}.
\end{equation}
If $a+b<\ell=m-1$, then $a+1<m-b$, so the two
retained segments in \eqref{eq:rectangle-dependence}
are disjoint. Thus there is an index $j$ such that
the term depends only on $(W,X_{\le j},Y_{>j})$,
and it can be included in $A_j$.
Also $u(x)$ depends only on $(W,X_{\le N})$,
and $v(y)$ only on $(W,Y_{\ge1})$; they can
be assigned to $A_N$ and $A_0$.

\emph{Bad corrections.}
Consider a term $\beta_Rh_R$ with $\beta_R\ne0$ and side lengths
$2^{-a},2^{-b}$. Put
\[
 j=\max\{i:p_i\le a+1\},\qquad
 k=\min\{i:p_i\ge m-b\},
\]
with the conventions $j=0$ or $k=N+1$ when the corresponding
set is empty. By \eqref{eq:rectangle-dependence}, the term is
measurable with respect to $(W,X_{\le j},Y_{\ge k})$.
If $j<k$, then $Y_{\ge k}$ is a subcollection of $Y_{>j}$,
so the term depends only on $(W,X_{\le j},Y_{>j})$ and is
included in $A_j$. If $k\le j$, it is included in $C_{k,j}$.

This completes the decomposition of $G$ in the form
\eqref{eq:ordered}. It remains to define the bad events
$\Omega_i$ and verify their measurability, probability bounds,
and the vanishing condition \eqref{eq:crossing-vanish}.

For each retained index $i$, let $\Omega_i$ be the event that the
principal rectangle of shape
\[
 |I|=2^{-(p_i-1)},\qquad |J|=2^{-(m-p_i)}
\]
containing $(x,y)$ is bad. Membership in that rectangle
depends only on the first $p_i-1$ digits of $x$ and the first
$m-p_i$ digits of $y$. Therefore
\begin{equation}\label{eq:events-measurable}
 \Omega_i\in\sigma(W,X_{<i},Y_{>i}),\qquad
 \Prob(\Omega_i)=2^{-\ell}\#\B_{\ell,p_i-1}\le\delta.
\end{equation}
If $k\le i\le j$, then $p_i\le a+1$ and $p_i\ge m-b$.
Consequently, $R$ is contained in its unique principal ancestor
of the shape associated with $p_i$. Since $\beta_R\ne0$, $R$ is
bad, so this ancestor must be bad by heredity.
It follows that
\[
 \ind_{\Omega_i^c}\beta_Rh_R=0\qquad(k\le i\le j).
\]
Summing over all terms assigned to $C_{k,j}$ gives
\[
 \ind_{\Omega_i^c}C_{k,j}=0
 \qquad(k\le i\le j).
\]
This verifies \eqref{eq:crossing-vanish}. All remaining dependence
on the unretained digits is included in $W$.

Lemma~\ref{lem:averaging} now applies to $G$. Since $N\ge m/(K+1)$,
\begin{equation}\label{eq:core-bound}
 \lambda_2\left\{|G|\ge
       \frac{\sqrt{\theta m}}{4\sqrt{K+1}}\right\}\ge\frac \theta{25}.
\end{equation}
For the tail, orthogonality gives $\|H_{\ell+s}\|_2^2=m+s$,
and different levels are orthogonal. Hence
\begin{equation}\label{eq:tail-energy}
 \|\mathcal R_K\|_2^2\le\sum_{s=K+1}^{\infty}(m+s)4^{-s}
       \le m\sum_{s=K+1}^{\infty}(s+1)4^{-s}.
\end{equation}
With $z_K=\sqrt{\theta m}/(8\sqrt{K+1})$, Chebyshev's inequality gives
\[
 \lambda_2\{|\mathcal R_K|\ge z_K\}\le
 \frac{64(K+1)}\theta\sum_{s=K+1}^{\infty}(s+1)4^{-s}.
\]
Since
\[
 \{|G|\ge2z_K\}\setminus\{|\mathcal R_K|\ge z_K\}
 \subseteq\{|f|\ge z_K\},
\]
these estimates and the identity
\[
 \sum_{s=K+1}^{\infty}(s+1)4^{-s}
 =\frac{3K+7}{9\cdot4^K}
\]
prove \eqref{eq:variable-bound}.

If $\delta\le1/2$, the hypotheses also hold with
$\delta=1/2$. Taking $K=10$ in
\eqref{eq:variable-bound}, we obtain
\[
 \lambda_2\left\{|f|\ge
    \frac{\sqrt m}{8\sqrt{22}}\right\}
 \ge\frac1{50}-\frac{407}{73728}
 >\frac1{100}.
\]
This proves \eqref{eq:explicit-haar}.
\end{proof}

\section*{Declaration of generative AI and AI-assisted technologies
in the manuscript preparation process}

The author formulated the main ideas and the proof strategy,
including the representation of Haar functions by random variables
and the use of an averaging lemma. Over several years, the author
made repeated attempts to implement this approach but was unable
to resolve the technical difficulties needed for a complete proof
of the result now stated as Lemma~\ref{lem:averaging}.
The part of the argument dealing with the incorporation of the
coarse Haar levels had already been developed by the author but
remained unpublished. ChatGPT (OpenAI) was used to develop the
technical details needed to incorporate the finer levels and
complete the proof within this framework, and to assist with
literature searches, the exposition, and the preparation of the
\LaTeX{} source. The author critically reviewed and revised the
generated material, verified the proofs and references, and takes
full responsibility for the content of the manuscript.

\end{document}